\documentclass[10pt]{article}

\usepackage{geometry}
\RequirePackage[OT1]{fontenc}
\RequirePackage{amsthm,amsmath,amssymb}
\RequirePackage[colorlinks,citecolor=blue,urlcolor=blue]{hyperref}
\usepackage{stmaryrd}
\usepackage[sans]{dsfont}
\usepackage{authblk}

\usepackage[english]{babel}
\usepackage{latexsym}
\usepackage{bbm}
\usepackage[mathcal]{euscript}
\usepackage{graphicx}
\usepackage{pstricks}
\usepackage{pstricks-add}
\usepackage{pst-plot}
\usepackage{enumerate}
\usepackage{multicol}
\usepackage{subfigure}
\usepackage{mathrsfs}
\newcommand{\e}{\varepsilon}   
\newcommand{\F}{\mathcal{F}}    
 
\newcommand{\ind}{\mathbbm{1}} 

\newcommand{\NN}{\mathbb{N}}
\newcommand{\RR}{\mathbb{R}}    

\newcommand{\VS}{\mathscr{V}}
\newcommand{\x}{\mathbf{x}}

\numberwithin{equation}{section}
\theoremstyle{plain}
\newtheorem{theorem}{Theorem}[section]
\newtheorem{lemma}[theorem]{Lemma}
\newtheorem{proposition}[theorem]{Proposition}
\newtheorem{corollary}[theorem]{Corollary}
\newtheorem{remark}{Remark}

\theoremstyle{definition}
\newtheorem{definition}{Definition}[section]

\begin{document}

\title{Approximative compactness in Böchner spaces}
\author{Guillaume Grelier\thanks{Departamento de Matemáticas, Universidad de Murcia, 
Campus de Espinardo, 30100 Espinardo, Murcia, Spain email: g.grelier@um.es}
\& Jaime San Mart\'in\thanks{CMM-DIM;  Universidad de Chile; UMI-CNRS 2807; 
Casilla 170-3 Correo 3 Santiago; Chile. email: jsanmart@dim.uchile.cl}}

\maketitle

\begin{abstract}
For any $p\in[1,\infty)$, we prove that the set of simple functions taking at most $k$ different values is proximinal in B\"ochner spaces $L^p(X)$
whenever $X$ is a dual Banach space with $w^*$-sequentially compact unit ball. With additional properties on $X$ and its norm, we show these sets 
are approximatively $w^*$-compact for $p\in(1,\infty)$ and even approximatively norm-compact under stronger hypothesis.
\end{abstract}



\section{Introduction}

In this paper, we study the approximation of measurable functions
by simple functions taking at most $k$ values for $k\ge 1$. This
problem has important consequences in multiple applications, 
where for example, one seeks for reduction of dimensionality, 
among many others. \\

We consider a measure space $(\Omega,\F, \mu)$ and a Banach space $X$. The set of simple functions will 
be denoted by $\mathcal S(\Omega,\F,\mu,X)$ (or simply $\mathcal S(X)$ if the measure space is clear from the context), that is 
$$
\mathcal S(\Omega,\F,\mu,X)=\left\{\sum\limits_{i=1}^n x_i \ind_{A_i}\ :\ n\geq 1,\ x_1,...,x_n\in X,\ A_1,...,A_n\in\F\right\}.
$$ 
We recall that a function $f:\Omega\to X$ is strongly measurable if it is a pointwise 
limit of a sequence of simple functions. The definition of the Bochner spaces is the following:  
$$
L^p(\Omega,\F,\mu,X)=\left\{f:\Omega\to X\ :\ f\text{ is strongly measurable and }\int_\Omega\|f\|^pd\mu<\infty\right\}\ \text{ if } 1\le p<\infty,
$$
$$
L^\infty(\Omega,\F,\mu,X)=\Big\{f:\Omega\to X\ :\ f\text{ is strongly measurable such that }\exists r>0\ \mu(\{\|f\|>r\})=0\Big\}.
$$ 
Endowed with the norm defined by
\begin{equation}
\|f\|_p=
    \begin{cases}
        \left(\int_\Omega\|f\|^pd\mu\right)^{1/p} & \text{if } 1\le p<\infty\\
       \inf\left\{r\geq 0\ :\ \mu(\{\|f\|>r\}\right)=0 & \text{if } p=\infty,
    \end{cases}
\end{equation}
$L^p(\Omega,\F,\mu,X)$ becomes a Banach spaces. Again, if no confusion is possible, we just write $L^p(X)$ 
instead of $L^p(\Omega,\F,\mu,X)$. For more information about Bochner spaces, we refer the reader to \cite{Diestel}.\\

In what follows, for any $1\leq p\leq\infty$ and $k\ge 1$, we denote by 
$\mathscr{G}_{p,k}(\Omega,\F, \mu,X)$, or simply
$\mathscr{G}_{p,k}(X)$, the set of simple functions given by
$$
\mathscr{G}_{p,k}(X)=\left\{\sum\limits_{i=1}^l x_i \ind_{A_i}\in L^p(\Omega,\F,\mu) :\ l\le k,\
\{A_i\}_{1\le i\le l}\subset\F \hbox{ partition of } \Omega,\ x_1,...,x_l\in X \right\}.
$$
If $f\in L^p(\Omega,\F,\mu,X)$, the distance between $f$ and $\mathscr{G}_{p,k}(X)$ is 
denoted by $\mathscr D_{p,k}(f)$, that is 
$$
\mathscr D_{p,k}(f)=\inf_{g\in\mathscr{G}_{p,k}}\|f-g\|_p.
$$ 

Before stating our main results, we recall some notions from approximation theory. Let $Z$ be a
Banach space and let $K$ be a subset of $Z$. 
The \textit{metric projection} on $K$ is the multi-valued mapping 
$P_K:Z\rightrightarrows K$ defined by $P_K(z)=\{w\in K\ :\ \|z-w\|=d(z,K)\}$ (where $d(A,B)$ is the 
distance between two subsets $A$ and $B$ of $Z$). If $z\in Z$ and if $(w_n)_n\subset K$ is a sequence such that 
$\|z-w_n\|\to d(z,K)$, we say that $(w_n)_n$ is a \textit{minimizing sequence for $z$} and if $w\in K$ is such that 
$\|z-w\|=d(z,K)$, we say that $w$ is a \textit{minimizer} for $z$. 
Consider $\tau$ a regular mode of convergence (resp. sequential convergence) We say that:
\begin{enumerate}
    \item[(a)] $K$ is \textit{proximinal} if $P_K(z)$ is not empty for all $z\in Z$;
    \item[(b)] $K$ is \textit{Chebyshev} if $P_K(z)$ is a singleton for all $z\in Z$;
    \item[(c)] $K$ is \textit{approximatively $\tau$-compact} (resp. \textit{approximatively sequentially $\tau$-compact}) 
    if for any $z\in Z$, any minimizing sequence for $z$ admits a $\tau$-convergent subnet (resp. subsequence) to a point in $K$.
\end{enumerate}
For the definition of regular mode of convergence and for more information about approximatively compactness, we refer the reader to \cite{Deutsch}.\\

In our case if $g\in\mathscr{G}_{p,k}(X)$ is such that $\|f-g\|_p=\mathscr D_{p,k}(f)$, we say that $g$ is a \textbf{minimizer} for $f$ in 
$\mathscr{G}_{p,k}(X)$.
A sequence $(g_n)_n\subset\mathscr{G}_{p,k}(X)$ is called a \textbf{minimizing sequence} for $f$ if 
$$
\lim\limits_{n\to \infty} \|f-g_n\|_p=\mathscr D_{p,k}(f).
$$

The main results of this article are the following two theorems, which generalize the results in \cite{G-JSM} for real valued functions.

\begin{theorem}
\label{the:mainresult I}
Let $(\Omega,\F,\mu)$ be a measure space, $p\in[1,+\infty)$, $k\geq 1$ and $X$ be a dual Banach space with $w^*$-sequentially compact unit ball. Then $\mathscr{G}_{p,k}(X)$ is proximinal.
\end{theorem}

Note that the class of dual Banach spaces with $w^*$-sequentially compact unit ball is very large. 
In fact, it includes dual spaces of a separable space (and more generally of a 
Gâteaux differentiability space, see Theorem 2.1.2 in \cite{Fabian}) and reflexive spaces for example. If $X$ is a finite dimensional Banach space, the previous result is also true for $p=\infty$ (see Remark \ref{finite_dim}). \\

The Radon-Nikodym property (in short, RNP) is a well-known property of Banach spaces. For more information about it, we refer the reader to \cite{Bourgin} for geometric characterizations of the RNP (see Theorem 2.3.6) and to \cite{Pisier} for characterizations in terms of martingales (see Theorem 2.9). If $X$ is a Banach space, it is worth pointing out that $X^*$ has the RNP if and only if $X$ is a Asplund space (see Theorem 5.7 in \cite{Phelps}).

\begin{theorem}
\label{the:mainresult II}
Let $(\Omega,\F,\mu)$ be a measure space, $p\in(1,+\infty)$, $k\ge 1$ and $X$ be the dual space of a Banach space with the RNP
and assume that the unit ball of $X$ is $w^*$-sequentially compact. Suppose that the norm 
of $X$ is G\^ateaux differentiable. Then $\mathscr{G}_{p,k}(X)$ is approximatively $w^*$-compact.
\end{theorem}

Since reflexivity implies RNP (see Corollary 2.15 in \cite{Pisier} for instance), 
the previous result applies to any reflexive Banach space endowed with a G\^ateaux differentiable norm.\\

With an additional hypothesis, we obtain a stronger result:

\begin{corollary}
\label{cor:norm-compact}
Let $(\Omega,\F,\mu)$ be a measure space, $p\in(1,+\infty)$, $k\ge 1$ and $X$ be a Banach space with a uniformly 
convex and G\^ateaux differentiable norm. Then $\mathscr{G}_{p,k}(X)$ is approximatively norm-compact.
\end{corollary}

Enflo's theorem (see \cite{Enflo}) states that a Banach space $X$ is superreflexive if and only if 
it admits an equivalent uniformly convex norm. In this case, it is always possible to find an equivalent 
norm on $X$ which is uniformly convex and G\^ateaux differentiable (even uniformly Fréchet differentiable, 
see Theorem 9.14 in \cite{BST}). Moreover, using a technique called Asplund averaging (see \cite{DGZ}, p.52), it can be proved that the sets of equivalent norms being uniformly convex and G\^ateaux differentiable is dense in the set of all equivalent norms.\\

In Section \ref{sec:II}, we introduce two concepts that play an important role
in the proofs of the main results. One of them is the Voronoi cells associated to a 
finite subset of $X$ (see Definition \ref{def:Voronoi cells}). 
The other one is the $p$-th mean of a function $f\in L^p$ over
a measurable set $A$. 
In Section \ref{sec:proof main I}, we prove Theorem \ref{the:mainresult I}. In Section
\ref{sec:particular}, we provide results about special forms for minimizers
(see Definition \ref{def:f Voronoi form} and \ref{def:special Voronoi form}).
In Theorem \ref{the:special Voronoi}, 
we show that all minimizers are in special $f$-Voronoi form, 
when $1\le p<\infty$. In Proposition \ref{pro:simple special}, assuming
the norm in $X$ is G\^ateaux differentiable and $1<p<\infty$, we show that every minimizer is in simple special $f$-Voronoi form. 
This concept, is crucial to prove Theorem \ref{the:mainresult II} in Section \ref{sec:proof main II} (see also there the proof of
Corollary \ref{cor:norm-compact}).
Finally, in Appendix \ref{app:pmeans}, we summarize some properties about $p$-th means.

\section{Some notations}
\label{sec:II}

We believe that our notation is quite standard. For example, the cardinal of a set $I$ is denoted by $|I|$. 
The closure, the interior and the boundary of a subset $A$ of a Banach space $X$ are respectively denoted by $\overline{A}$, $\text{int}(A)$ and $\partial A$. The distance between two subsets $A$ and $B$ of $X$ will be denoted by $d(A,B)$. \\

Now we introduce some definitions and specific notations that will be used in this paper.

\begin{definition}
\label{def:Voronoi cells}
Let $X$ be a Banach space and $\x=\{x_1,...,x_k\}\subset X$ be a finite set of different points. The \textbf{Voronoi cells associated
to $\x$} is the finite collection $\VS(x)=\{V_i\}_{1\leq i\leq k}$ given by
$$
V_i=\left\{y\in X\ :\  \|y-x_i\|\le \min_{j\neq i}\|y-x_j\|\right\}.
$$
\end{definition}

Notice that each $V_i$ is closed and it is star convex with respect to $x_i$. Also $x_i$ is in the interior of
$V_i$ and $y\in \partial V_i$ if and only if $y\in V_i$ and there exists $j\neq i$ such that $\|y-x_i\|=\|y-x_j\|$. Therefore
$$
\hbox{int} (V_i)=\{y\in X: \|y-x_i\|< \|y-x_j\|, \hbox{ for all } j\neq i\}.
$$
This shows that $\hbox{int}(V_i)\cap V_j=\emptyset$ for all $i\neq j$.

\begin{definition}
\label{def:reduced form}
Let $(\Omega,\F,\mu)$ be a measure space and $X$ be a Banach space. 
Let $h=\sum\limits_{i=1}^k x_i \ind_{A_i}\in\mathcal S(X)$ a simple function. 
\begin{itemize}
    \item We say that $h$ is in \textbf{reduced form}
if the $x_i$'s are all different and $\{A_i\}_{1\leq i\leq k}$ is a measurable partition of $\Omega$ with sets 
of positive measure, except for a set of measure $0$, that is $\mu(A_i\cap A_j)=0$ 
for all $i\neq j$ and $\mu(\Omega \setminus \cup_i A_i)=0$.
\item We also define the \textbf{degree} of $h$ by 
$$
\deg(h)=\min\left\{l\geq 1\ :\ \exists\, g=\sum_{i=1}^ly_i\ind_{B_i}\in\mathcal S(X)\text{ such that } h=g\text{ a.e.}\right\}.
$$
\end{itemize}
\end{definition}

\begin{remark} 
\label{rem:1}
When $\mu$ is an infinite measure, $p<\infty$ and $h\in L^p(X)$ is in reduced form, then there exists a unique index
$1\le i_0\le \deg(h)$ so that $x_{i_0}=0$ and for all $1\le i\le \deg(h)$ with $i\neq i_0$, it holds $\mu(A_i)<\infty$. 
By reordering the terms in $h$ we assume always that $i_0=1$.
\end{remark}

\begin{definition}
\label{def:f Voronoi form}
Let $(\Omega,\F,\mu)$ be a measure space and $X$ be a Banach space. Let $f\in L^p(X)$ and 
$h=\sum\limits_{i=1}^k x_i \ind_{A_i}\in\mathcal S(X)$ in reduced form. 
Let $\x=\{x_1,...,x_k\}\subset X$ and $\VS(x)=\{V_i\}_{1\leq i\leq k}$. 
We say that $h$ is in {\bf \emph{$f$-Voronoi form}} if it holds that 
$A_i\subset f^{-1}(V_i)$ $\mu$-a.e. for all $1\leq i\leq k$. 
\end{definition}

Note that, since $\{A_i\}_{1\leq i\leq k}$ is a partition, we can assume further that $f^{-1}(\hbox{int}(V_i))\subset A_i$
holds $\mu$-a.e., for all $i$.

\begin{definition}
\label{def:pth mean} Let $(\Omega,\F,\mu)$ be a measure space and $X$ be a Banach space. 
Let $f\in L^p(X)$ and $A\in\F$. The function $M_p(f,A):X \to \RR_+$ is defined by
$$
M_p(f,A)(x)=\begin{cases}\int_A \|f(w)-x\|^p \ \mu(dw) &\hbox{ if } p<\infty\\
					\|(f-x)\ind_A\|_\infty	          &\hbox{ if } p=\infty
	            \end{cases},
$$ We denote by $\underline M_p(f,A)=\inf\limits_{y\in X} M_p(f,A)(y)$, the infimum
of $M_p(f,A)$. Given $\e\ge 0$, we say that $x\in X$ is an \textbf{$\e$-$p$-th mean} for $f$ in $A$ if
$$
M_p(f,A)(x)\le \underline M_p(f,A)+\e.
$$
In case $\e=0$, we simply say $x$ is a \textbf{$p$-th mean} of $f$ in $A$. 
\end{definition}

Some properties of the function $M_p(f,A)$ and the $\e$-$p$-th means used in this document are presented in the Appendix \ref{app:pmeans}.

\section{Existence of a minimizer. Proof of Theorem \ref{the:mainresult I}}
\label{sec:proof main I}

\begin{proposition}
\label{pro:bounded}
Let $(\Omega,\F,\mu)$ be a measure space, $X$ a Banach space, $p\in[1,\infty]$, $f\in L^p(X)$ and $k\geq 1$. 
Let $(h_n)_n\subset\mathscr G_{p,k}(X)$ a minimizing sequence for $f$. 
Then there exist a uniformly bounded sequence 
$(g_n)_n\subset\mathscr G_{p,k}$ and a subsequence $(h'_{n})_n$ of $(h_n)_n$ such that 
$\|g_n-h'_n\|_p\to 0$, in particular $\|f-g_n\|\to\mathscr D_{p,k}(f)$.
\end{proposition}

\begin{proof} 
The case $p=\infty$ is evident. So we assume that $p<\infty$.
Let us write $h_n$ in reduced form
$$
h_n=\sum\limits_{i=1}^{\ell(n)} x_{i,n} \ind_{A_{i,n}} \in \mathscr{G}_k(X),
$$
where $\ell(n)=\deg(h_n)\le k$. By considering 
a subsequence if necessary, we also assume $\ell(n)=\ell$ is fixed. By taking a further 
subsequence if necessary, we can and do assume that $(\|x_{i,n}\|)_n$ converges in $\overline{\mathbb R}$. 
Note that $(h_n)_n$ is a bounded sequence in $L^p(X)$. Define 
$$
I=\{i\in\{1,...,l\}\ :\ \|x_{i,n}\|\to\infty\}.
$$ 
For all $n\in\mathbb N$, we define 
$$
g_n=\sum_{i\notin I} x_{i,n} \ind_{A_{i,n}} + 0\ind_{B_n}\in \mathscr G_{p,k^*}(X)\subset \mathscr G_{p,k}(X),
$$ 
where $B_n=\bigcup_{i\in I}A_{i,n}$ and $k^*$ is defined by $k^*=\ell\le k$ if $I$ is empty and $k^*\le \ell-|I| +1\le k$ if not. Since $(h_n)_n$ is bounded and $\|h_n\|_p^p\geq\mu(B_n)\min_{i\in I}\|x_{i,n}\|^p$, 
it follows that $\mu(B_n)\to 0$. We have that 
$$
\|h_n-g_n\|_p^p=\sum\limits_{i\in I} \|x_{i,n}\ind_{A_{i,n}}\|_p^p\le 2^{p-1} 
\left( \sum\limits_{i\in I} \|(f-x_{i,n})\ind_{A_{i,n}}\|_p^p+\int_{B_n} \|f(w)\|^p\, d\mu(w)\right).
$$ 
Since $\int_{B_n} \|f(x)\|^p\, d\mu(x)\to 0$, it remains to show that 
$\sum\limits_{i\in I} \|(f-x_{i,n})\ind_{A_{i,n}}\|^p\to 0$.
We have that for all $n\in\mathbb N$
$$
\begin{array}{ll}
\mathscr{D}_{p,k}(f)^p&\hspace{-0.2cm}\le \|f-g_n\|_p^p\le
\|f-g_n\|^p+\sum_{i\in I} \|(f-x_{i,n})\ind_{A_{i,n}}\|_p^p\\
\\
&\hspace{-0.2cm}=\|f- h_n\|_p^p+\|f\ind_{B_n}\|_p^p\to\mathscr{D}_{p,k}(f)^p,
\end{array}
$$
proving that $\sum\limits_{i\in I} \|(f-x_{i,n})\ind_{A_{i,n}}\|_p^p\to 0$.
\end{proof}

\begin{proposition}
\label{pro:Voronoi_form}
Let $(\Omega,\F,\mu)$ be a measure space, $X$ a Banach space, $p\in [1,\infty]$, $f\in L^p(X)$ and $k\ge 1$. 
Suppose that there exists a minimizing sequence $(g_n)_n\subset\mathscr{G}_{p,k}(X)$ such that the reduced form of $g_n$ is given by
$$
g_n=\sum_{k=1}^\ell x_i\ind_{A_{i,n}},
$$ 
where $\deg(g_n)=\ell\leq k$ and $x_1,...,x_\ell\in X$. Then, there exists a minimizer $g\in\mathscr{G}_{p,\ell}(X)$ for $f$ in $f$-Voronoi form. 

\end{proposition}

\begin{proof} Since $g_n=\sum_{k=1}^\ell x_i\ind_{A_{i,n}}$ is in reduced form, we have the $x_i$'s are all different. Consider $\VS(\x)=\{V_i\}_{1\leq i\leq\ell}$ the Voronoi decomposition of $X$ associated to $\x=\{x_i\}_{i\le \ell}$. 
Define the following modification of $\VS(\x)$. Let $D_1=V_1$ and define recursively for $2\leq j\leq \ell$
\begin{equation}
\label{eq:4}
D_j=V_j\setminus \bigcup_{m<j} D_m
\end{equation}
Notice that $\hbox{int}(V_j)\subset D_j\subset V_j$ for all 
$j$ and that $(D_j)_{1\leq j\leq \ell}$ is a partition of $X$.
For all $i,j$ and all $w\in A_{i,n}\cap f^{-1}(D_j)$, we have 
\begin{equation}
\label{eq:Vor1}
\|f(w)-x_j\|\le \|f(w)-x_i\|,
\end{equation}
with equality if $i=j$.
Consider $g$ given by
$$
g=\sum\limits_{i,j} x_j \ind_{A_{i,n}\cap f^{-1}(D_j)}=\sum\limits_{j} x_j \ind_{f^{-1}(D_j)}.
$$
We notice that $g_n=\sum\limits_{i,j} x_i \ind_{A_{i,n}\cap f^{-1}(D_j)}$.
Then, when $p=\infty$, it holds
$$
\|f-g\|_\infty=\max_{i,j}\|(f-x_j)\ind_{A_{i,n}\cap f^{-1}(D_j)}\|_\infty\le 
\max_{i,j}\|(f-x_i)\ind_{A_{i,n}\cap f^{-1}(D_j)}\|_\infty
=\|f-g_n\|_\infty
$$
proving that $g$ is a minimizer. Next, assume that $p<\infty$. Integrating the inequality \eqref{eq:Vor1}, we obtain that
$$
\|f-g\|_p^p=\sum_{i,j} \int_{A_{i,n}\cap f^{-1}(D_j)} \|f(w)-x_j\|^p \,\mu(dw)\le 
\sum_{i,j} \int_{A_{i,n}\cap f^{-1}(D_j)} \|f(w)-x_i\|^p \,\mu(dw)=\|f-g_n\|_p^p,
$$
and then $g$ is a again a minimizer.

Now, let us show that $g$ is in $f$-Voronoi form. For that, consider $J=\{j: \mu(f^{-1}(D_j))>0\}$, 
which is not empty because $(f^{-1}(D_j))_{j\in J}$ is a partition except for a set of measure $0$.
Consider ${\bf y}=\{x_j\}_{j\in J}$ and the Voronoi cells associated $\VS({\bf y})=\{W_j\}_{j\in J}$. Since $D_j\subset V_j\subset W_j$, we have 
$$
g=\sum\limits_{j\in J} x_j \ind_{f^{-1}(D_j)}
$$
is in $f$-Voronoi form.
\end{proof}

\begin{proof} ({\bf Theorem \ref{the:mainresult I}})
By Proposition \ref{pro:bounded}, there exists a uniformly bounded mimimizing sequence $(h_n)_n$. We write $h_n$ in reduced form
$$
h_n=\sum\limits_{i=1}^{l(n)} x_{i,n} \ind_{A_{i,n}} \in \mathscr{G}_{p,k}(X),
$$
where $l(n)\le k$. For all $1\leq i\le l$, the sequence $(x_{i,n})_{n}$ is bounded. 
By $w^*$-sequentially compactness and taking a subsequence if necessary, we can suppose that $\ell(n)=\ell<k$ is constant and
$$
x_{i,n}\xrightarrow[]{w^*} x_i.
$$
For all $n\in\mathbb N$, we define 
$$
g_{n}=\sum\limits_{i=1}^l x_i \ind_{A_{i,n}}\in \mathscr{G}_{p,k}(X),
$$
where we assume without loss of generality that all $(x_i)_{i\le \ell}$ 
are different. In case $\mu(\Omega)=\infty$, we can suppose that
$x_1=0$ and $\mu(A_{i,n})<\infty$ for all $2\le i\le l$ (see Remark \ref{rem:1}).
\medskip

To continue with the proof, we first assume that $\mu(\Omega)<\infty$. Let us prove that $(g_n)_n$ admits a 
minimizing subsequence. Since $(g_n)_n$ and $(h_n)_n$ are uniformly bounded on a finite measure space, 
it is easy to see that the sequence $(\|h_n-f\|^p-\|g_n-f\|^p)_n$ is bounded below by 
a integrable function. Then, Fatou's lemma implies that
$$
\liminf_n\|h_n-f\|_p^p-\|g_n-f\|_p^p\geq\sum_{i=1}^l\int\liminf_n\alpha_{i,n}(\omega)d\mu(\omega),
$$ 
where $\alpha_{i,n}(\omega)=\left(\|x_{i,n}-f(\omega)\|^p-\|x_i-f(\omega)\|^p\right)\ind_{A_{i,n}}(\omega)$. Let $i\in\{1,...,l\}$
and take $\omega\in\Omega$. We define $I(\omega)=\{n\in\mathbb N\ :\ \omega\in A_{i,n}\}$. The accumulation points of the sequence $(\alpha_{i,n}(\omega))_{n\in I(\omega)}$ are non-negative by $w^*$-semi-continuity of the norm and since 
$(\alpha_{i,n}(\omega))_{n\in I(\omega)^c}$ is the null sequence, we deduce that the accumulation points of $(\alpha_{i,n}(\omega))_n$ are also non-negative. It follows that $\liminf_n\alpha_{i,n}(\omega)\geq 0$ for 
any $i\in\{1,...,l\}$ and any $\omega\in\Omega$. Then, we obtain that 
$$
\liminf_n\|h_n-f\|_p^p-\|g_n-f\|_p^p\geq 0,
$$ 
and it follows that $(g_n)_n$ admits a minimizing subsequence. We conclude this case using Proposition \ref{pro:Voronoi_form}. 

Now, we assume $\mu(\Omega)=\infty$. For every $\e>0$, consider $G_\e=\{w: \|f(w)\|\ge \e\}$, which is a set
of finite measure. By Fatou's Lemma, we have
$$
\begin{array}{ll}
\mathscr D_{p,k}(f)^p\hspace{-0.2cm}&=\lim\limits_n \|f-h_n \|_p^p\ge \liminf\limits_n \|(f-h_n)\ind_{G_\e}\|_p^p\\
\hspace{-0.2cm}&=\liminf\limits_n \left(\|(f-h_{n})\ind_{G_\e}\|_p^p-\|(f-g_{n})\ind_{G_\e}\|_p^p
+\|(f-g_{n})\ind_{G_\e}\|_p^p\right)\\
\hspace{-0.2cm}&\ge \liminf\limits_n \left(\|(f-h_{n})\ind_{G_\e}\|_p^p-\|(f-g_{n})\ind_{G_\e}\|_p^p \right)+
\liminf\limits_n\|(f-g_{n})\ind_{G_\e}\|_p^p\\
\hspace{-0.2cm}&\ge \liminf\limits_n\|(f-g_{n})\ind_{G_\e}\|_p^p
\end{array}
$$ 
For $\e=2^{-r}$ with $r\in \NN$, take $n_r$ so that with $B_r=G_{2^{-r}}$
$$
\|(f-g_{n_r})\ind_{B_r}\|_p^p\le \mathscr D_{p,k}(f)^p +2^{-r}
$$
Recall that $x_1=0$, that is $g_n(w)=0$ for all $w\in A_{1,n}$. Note that
$$
\tilde g_{n_r}=g_{n_r}\ind_{B_r}=\sum\limits_{i=2}^l x_i\ind_{A_{i,n_r}\cap B_{r}}+
0 \ind_{A_{1,n_r}\cup B^c_{r}}\in \mathscr{G}_{p,l}(X)\subset \mathscr{G}_{p,k}(X),
$$
and
$$
\|f-\tilde g_{n_r}\|_p^p=\|(f-g_{n_r})\ind_{B_r}\|_p^p+
\int_{B^c_{r}\setminus A_{\zeta,n_r}} \!\!\!\!\!\|f(w)\|^p \, \mu(dw)
\le \mathscr D_{p,k}(f)^p +2^{-r}+\int_{B^c_{r}} \|f(w)\|^p \, \mu(dw).
$$
Notice that $\|f\|^p\ind_{B^c_{r}}=\|f\|^p\ind_{\|f\|\le 2^{-r}}$ converges pointwise to $0$ and it 
is dominated by $\|f\|^p$. This shows that the sequence $(\tilde g_{n_r})_r$ is 
a minimizing sequence. We conclude again by Proposition \ref{pro:Voronoi_form}.
\end{proof}

\begin{remark}\label{finite_dim}
Note that if $X$ is finite-dimensional then the previous result also holds for $p=\infty$. In fact, in this case, we have that $x_{i,n}\xrightarrow[]{} x_i$. It follows that $$\|g_n-f\|_\infty\leq\|g_n-h_n\|_\infty+\|h_n-f\|_\infty=\max_{1\leq i\leq l}\|x_{i,n}-x_i\|+\|h_n-f\|_\infty\to\mathscr{D}_{\infty,k}(f),$$ proving that $(g_n)_n$ is a minimizing sequence. Then the conclusion follows by Proposition \ref{pro:Voronoi_form}.
\end{remark}

\section{Particular form of a minimizer}
\label{sec:particular}

\begin{definition}
\label{def:special Voronoi form}
Let $(\Omega,\F,\mu)$ be a measure space, $X$ a Banach space. Let $f\in L^p(X)$ and $h=\sum\limits_{i=1}^k x_i \ind_{A_i}$ in $f$-Voronoi form. 
We say that $h$ is in {\bf \emph{special $f$-Voronoi form}} if $x_i$ is a $p$-th mean 
of $f$ in $A_i$ for all $1\leq i\leq k$. Moreover, we say that $h$ is in 
{\bf \emph{simple special $f$-Voronoi form}} if for all $1\leq i\leq k$ it holds $\mu(f^{-1}(\partial V_i))=0$, that is $A_i=f^{-1}(\hbox{int}(V_i))$ holds $\mu$-a.e.
\end{definition}

\begin{proposition}
\label{pro:f in G_pk}
Let $X$ be a Banach space, $p\in[1,\infty)$, $k\geq 1$ and $f\in L_p(X)$. 
Assume $f$ has a minimizer $g\in \mathscr{G}_{p,k}(X)$ so that $\deg(g)=r< k$. 
Then $f\in \mathscr{G}_{p,r}$. 
\end{proposition}

\begin{proof}
By Proposition \ref{pro:Voronoi_form}, we can assume that $g$ is in special 
$f$-Voronoi form 
$$
g=\sum\limits_{i=1}^r x_i \ind_{f^{-1}(D_i)}.
$$
Since $f$ is the pointwise limit of simple functions, we can and do suppose that $f$ has separable range. 
Then defining $Y=\overline{\text{span}}(f(\Omega)\cup\{x_1,...,x_r\})$ 
if necessary, we can also assume that $X$ is separable. Let us prove that $f\in \mathscr{G}_{r}$. 
By contradiction, suppose that 
$f\notin \mathscr{G}_r$. In particular, we have 
$$
\mu(f^{-1}(\{x_1,...,x_r\}^c))>0.
$$
Note that 
$$
f^{-1}(\{x_1,...,x_r\}^c)=\bigcup_{t\in \NN,t\ge 1}f^{-1}\left(\bigcup_{i=1}^{r}D_i\setminus B(x_i,1/t)\right),
$$
which implies that there exists $t\in \NN$ and $t\ge 1$ such that 
$$
\mu\left(f^{-1}\left(\bigcup_{i=1}^{r}D_i\setminus B(x_i,1/t)\right)\right)>0.
$$ 
Since $x_i$ lies in the interior of $D_i$, we can and do assume that $B(x_i,1/t)\subset D_i$ 
for all $1\leq i\leq r$. Consider an index $1\leq j\leq r$, for which 
$$
\mu\left(f^{-1}\left(D_j \setminus B(x_j,1/t)\right)\right)>0.
$$
Let $(y_n)_n$ be a dense sequence in $A:=D_j \setminus B(x_j,1/t)$. Since 
$f^{-1}(A)=\bigcup_n f^{-1}(A)\cap f^{-1}(B(y_n,\frac{1}{2t}))$, 
there exists $n\in\mathbb N$ such that $\mu(C)>0$ where 
$$
C=f^{-1}\left(A\cap B\left(y_n,\frac{1}{2t}\right)\right)=
f^{-1}(D_j \setminus B(x_j,1/t))\cap f^{-1}\left(B\left(y_n,\frac{1}{2t}\right)\right)\subset f^{-1}(D_j).
$$
Observe that, for $w\in C$ we have $\|f(w)-y_n\|<\frac{1}{2t}\le \frac12 \|f(w)-x_j\|$.
Define
$$
\tilde g=\sum\limits_{1\leq i\leq r,\ i\neq j} x_i \ind_{f^{-1}(D_i)}+x_j \ind_{f^{-1}(D_j)\setminus C}+y_n \ind_C \in \mathscr{G}_{p,r+1}\subset \mathscr{G}_{p,k}.
$$
One has that
$$ 
\begin{array}{ll}
\|f-g\|_p^p\hspace{-0.2cm}&=\|f-\tilde g\|_p^p+\int_C\|f(w)-x_j\|^p-\|f(w)-y_n\|^p\ \mu(dw)\\
\\
\hspace{-0.2cm}&\geq\|f-\tilde g\|_p^p+
(1-2^{-p})\int_C \|f(w)-x_j\|_p^p \ \mu(dw)\\
\\
\hspace{-0.2cm}&\ge \|f-\tilde g\|_p^p+(1-2^{-p})t^{-p}\mu(C),
\end{array}
$$
which contradicts the minimality of $g$ and concludes the proof. 
\end{proof}

\begin{theorem}
\label{the:special Voronoi}
Let $(\Omega,\F,\mu)$ be a measure space, $X$ a Banach space, $p\in[1,\infty)$, $k\geq 1$ and $f\in L_p(X)$. Then, all minimizers $h\in \mathscr{G}_{p,k}(X)$ for $f$ have a reduced form in special $f$-Voronoi form and fulfill that $\deg(h)=k$ whenever $f\notin \mathscr{G}_{p,k}(X)$.
\end{theorem}

\begin{proof}
When $f\in \mathscr{G}_{p,k}(X)$ the result follows directly. So, assume that $f\notin \mathscr{G}_{p,k}(X)$. Assume that $h=\sum\limits_{i=1}^k x_i \ind_{A_i}$ is a minimizer. From what we have proved, we have
all $x_1,...,x_k$ are different and $\mu(A_i)>0$ for all $i$, that is, $h$ has degree $k$. 
Consider $\VS(\x)=\{V_i\}_{1\leq i\leq k}$ the 
Voronoi cells associated to $\x=(x_1,...,x_k)$. Assume by contradiction that $\mu(A_i\setminus f^{-1}(V_i))>0$ for some $i$. 
Consider $Z_j=V_j\setminus V_i$ for all $j$ and we refine $(Z_j)_j$ to get a partition of $V_i^c$ as follows
$$
D_1=Z_1,\hbox{ and } D_j=Z_j\setminus \bigcup\limits_{1\le r<j} D_r, j=2,...,k
$$
Notice that $D_i=Z_i=\varnothing$, $A_i\setminus f^{-1}(V_i)=\bigcup\limits_{j\neq i} A_i \cap f^{-1}(D_j)$,
and $\bigcup\limits_{j\neq i} D_j=V_i^c$. Then for
all $w\in A_i\setminus f^{-1}(V_i)$ it holds
$$
\|f(w)-x_i\|^p=\sum_{j\neq i} \|f(w)-x_i\|^p \ind_{A_i\cap f^{-1}(D_j)}(w)> 
\sum_{j\neq i} \|f(w)-x_j\|^p \ind_{A_i\cap f^{-1}(D_j)}(w)
$$
This shows that
$$
g=x_i \ind_{A_i\cap f^{-1}(V_i)}+\sum\limits_{j\neq i} x_j\ind_{A_j \cup (A_i\cap f^{-1}(D_j))} \in \mathscr{G}_{p,k},
$$
satisfies
$$
\|f-g\|_p^p <\|f-h\|_p^p,
$$
which is a contradiction and then $A_i\subset f^{-1}(V_i)$ holds $\mu$-a.s. proving that $g$ is in $f$-Voronoi form.

Now, we prove $h$ is in special $f$-Voronoi form. Assume that for some 
$j_0\in J$, $x_{j_0}$ is not a $p$-th mean for $f$ in $f^{-1}(D_{j_0})$, then for some
$\e>0$ it holds
$$
\underline M_p(f,A)+\e\le M_p(f,A)(x_{j_0}),
$$ 
and then there exists $x_{j_0}(\e)$ such that
$$
M_p(f,A)(x_{j_0}(\e))+\e/2\le M_p(f,A)(x_{j_0}),
$$ 
which will lead to a contradiction. Indeed, denote by 
$$
\tilde h=\sum\limits_{j\in J,j\ne j_0} x_j \ind_{f^{-1}(D_j)}+x_j(\e) \ind_{f^{-1}(D_{j_0})} \in\mathscr{G}_{p,k}(X),
$$
then
$$
\|f-h\|_p^p-\|f-\tilde h\|_p^p=\int_{f^{-1}(D_{j_0})}(\|f(w)-x_{j_0}\|_p^p-\|f(w)-x_{j_0}(\e)\|_p^p)\ \mu(dw)>\e/2,
$$
contradicting the minimality of $h$.
\end{proof}

\begin{lemma}\label{lemma1}
Let $(\Omega,\F,\mu)$ be a measure space, $X$ a Banach space, $p\in[1,\infty]$, $k\geq 1$ and $f\in L_p(X)$. 
Suppose that $h=\sum\limits_{i=1}^s x_i\ind_{A_i}$ is a minimizer for $f$ in reduced form in $\mathscr{G}_{p,k}(X)$, with
$\deg(h)=s\le k$. Consider
\begin{itemize}
    \item $\x=\{x_1,...,x_s\}$ and $\VS(\x)=\{V_i\}_{1\leq i\leq s}$;
    \item a measurable set $Z\subset\partial V_1$;
    \item $D_1(Z)=Z\cup \hbox{int}(V_1)$ and $D_i(Z)=V_i\setminus \bigcup\limits_{r=1}^{i-1} D_r(Z)$ for all $2\leq i\leq s$.
\end{itemize}
Then $h_Z=\sum\limits_{i=1}^{s} x_i \ind_{f^{-1}(D_i(Z))}$ is a minimizer for $f$ in $\mathscr{G}_{p,k}(X)$.
\end{lemma}

\begin{proof}
Notice that $\hbox{int}(V_i)\subset D_i(Z)\subset V_i$ and then if $w\in A_i\cap f^{-1}(D_j(Z))$ it holds that
$$
\|f(w)-x_j\|\le \|f(w)-x_i\|.
$$
So, if $p<\infty$, we get
$$
\|f-h\|_p^p=\sum\limits_{i,j} \|(f-x_i)\ind_{A_i\cap f^{-1}(D_j(Z))}\|_p^p\ge
\sum\limits_{i,j} \|(f-x_j)\ind_{A_i\cap f^{-1}(D_j(Z))}\|_p^p= \|f-h_Z\|_p^p,
$$
and then $g_Z$ is also a minimizer. If $p=\infty$, the conclusion also holds with minor adjustements. 
\end{proof}

\begin{proposition} 
\label{pro:simple special}
Let $(\Omega,\F,\mu)$ be a measure space. Assume that $X$ is a Banach space with a G\^ateaux differentiable 
norm $\|\hspace{0.05cm}\|$. Let $f\in L^p(X)$ with $1<p<\infty$. If 
$h$ is a minimizer for $f$ in $\mathscr{G}_{p,k}(X)$, 
then the reduced form of $h$ is in simple special $f$-Voronoi form.
\end{proposition}

\begin{proof} 
The result is obvious if $k=1$, so let us suppose that $k>1$. There exist $s\leq k$, $A_1,...,A_s\in\F$ and $x_1,...,x_s\in X$ such that the reduced form of $h$ is 
$$
h=\sum\limits_{i=1}^s x_i\ind_{A_i}.
$$
Let $\x=\{x_1,...,x_s\}$ and $\VS(\x)=\{V_i\}_{1\leq i\leq s}$.

Suppose first that $f\in \mathscr{G}_{p,k}(X)$. Then $f=h$ holds $\mu$-a.e. and we have that
$$
h=\sum\limits_{i=1}^s x_i\ind_{f^{-1}(\{x_i\})},.
$$ which is in special $f$-Voronoi form by the previous result. Since $\mu(f^{-1}(\{x_1,...,x_s\}^c))=0$, we have that $f^{-1}(\{x_i\})=f^{-1}(\text{int}(V_i))$ $\mu$-a.e. and then the reduced form of $h$ is in simple special $f$-Voronoi form.

Now we suppose that $f\notin \mathscr{G}_{p,k}(X)$. From Proposition \ref{the:special Voronoi}, we have that the reduced form of $h$ 
is in special $f$-Voronoi form and $s=k$. If for some $1\leq i\leq k$, it holds that $\mu(f^{-1}(\hbox{int}(V_i)))=0$, we could absorb the boundary
of $V_i$ into the other cells in order to get a minimizer $\tilde h$ with degree smaller than $k$, which
contradicts Proposition \ref{the:special Voronoi}, and then for all $1\leq i\leq k$ it holds
$$
\mu(f^{-1}(\hbox{int}(V_i)))>0.
$$
We need to prove that $\mu(f^{-1}(\partial V_i))=0$ for all $1\leq i\leq k$. 

Suppose that $\mu$ is finite. Assume that for some $1\leq i\leq k$ it holds $\mu(f^{-1}(\partial V_i))>0$. Without loss of generality, we can assume $i=1$. Consider $Y=\overline{span}\{f(\Omega),\x\}$, which is a separable closed subspace of $X$. Consider a dense countable set $(y_n)_n\subset \partial V_1\cap Y$. Since $x_1\in \hbox{int}(V_1)$, there exists $\e_0>0$ such that $B(x_1,4\varepsilon_0)\subset V_1$. Note that $\|x_1-y_n\|\geq 4\varepsilon_0$ for all $n\in\mathbb N$.
Since 
$$
\mu(f^{-1}(\partial V_1\setminus Y))=0,
$$ 
we have $\mu(f^{-1}(\partial V_1\cap Y))=\mu(f^{-1}(\partial V_1))>0$. Fix $\varepsilon\in(0,\varepsilon_0]$. There exists $n\in\mathbb N$ so that 
$$
\mu(f^{-1}(\partial V_1\cap B(y_n,\e)))>0.
$$
We consider $Z_1=\partial V_1\cap B(y,\e)$, where $y=y_n\in \partial V_1$ and $Z_2=\emptyset$. Both functions $h_1=h_{Z_1}$ and $h_2=h_{Z_2}$ 
are minimizers by Lemma \ref{lemma1}. Note that $\mu(f^{-1} (\hbox{int}(V_1)\cup Z_1))\geq\mu(f^{-1}(\hbox{int}(V_1)))>0$. It follows that $x_1$ is a $p$-th mean of $f$ in both $f^{-1}(\hbox{int}(V_1))$ and 
$f^{-1}(\hbox{int}(V_1)\cup Z_1)$ (see Proposition \ref{the:special Voronoi}). Defining $v=y-x_1$ and using Corollary \ref{cor:charac_pmean}, we deduce that the real functions
$$
\begin{array}{l}
R(a)=\int_{f^{-1}({\small\hbox{int}}(V_1)\cup Z_1)} \|f(w)-(x_1+av)\|^p \ \mu(dw),\\
\\
S(a)=\int_{f^{-1}({\small\hbox{int}}(V_1))} \|f(w)-(x_1+av)\|^p \ \mu(dw)\\
\\
T(a)=R(a)-S(a)=\int_{f^{-1}(Z_1)} \|f(w)-(x_1+av)\|^p \ \mu(dw)
\end{array}
$$
are convex and differentiable. Moreover, $R$ and $S$ have a minimum 
at $0$ and we have that
$$
\begin{array}{l}
0=R'(0)=\int_{{\small\hbox{int}}(V_1)\cup Z_1} p\|f(w)-x_1\|^{p-1} \partial_v\|f(w)-x_1\|\ \mu(dw)\\
\\
0=S'(0)=\int_{{\small\hbox{int}}(V_1)} p\|f(w)-x_1\|^{p-1} \partial_v \|f(w)-x_1\| \ \mu(dw).
\end{array}
$$
This shows that $T'(0)=0$. Since $T$ is convex, it follows that $0$ is also a minimum of $T$. Let us show that this is not possible. Choose $a\in (0,1)$ so that $a\|x_1-y\|>3\e$ and consider $x^*=x_1+a(y-x_1)$. If $w\in f^{-1}(Z_1)$, we have that $f(w)\in \partial V_1 \cap B(y,\e)$ and then 
$$
\begin{array}{lll}
a\|x_1-y\|+(1-a)\|x_1-y\|&=&a\|x_1-y\|+\|y-(x_1+a(y-x_1))\|=a\|x_1-y\|+\|y-x^*\|\\
&=&\|y-x_1\|\le \|f(w)-x_1\|+\e,
\end{array}
$$
which implies
$$
a\|x_1-y\|+\|f(w)-x^*\|\le a\|x_1-y\|+\|y-x^*\|+\e= \|y-x_1\|+\e\le \|f(w)-x_1\|+2\e.
$$
It follows that for all $w\in f^{-1}(Z_1)$
$$
\|f(w)-x^*\|-\e\le \|f(w)-x_1\|.
$$
This implies that 
$$
T(a)<T(0),
$$
which is a contradiction. Then $\mu(f^{-1}(\partial V_1))=0$ and the result follows in case $\mu$ is finite.


Now, we assume $\mu(\Omega)=\infty$. By Remark \ref{rem:1}, we can assume that $x_1=0$. Let $Z=\emptyset$  and define
$$
h_Z=\sum\limits_{i=1}^{k} x_i \ind_{f^{-1}(D_i(Z))},
$$
as in the previous lemma. Then $h_Z$ is also a minimizer, $\mu(f^{-1}(D_1(Z)))=\infty$ and for all $j\ge 2$ 
we have 
$$
0<\mu(f^{-1}(D_j(Z)))<\infty.
$$
The first inequality holds thanks to Theorem \ref{the:special Voronoi} since $\deg(g_Z)=k$, and the second inequality holds
because $x_j\neq 0$ is a $p$-th mean of $f$ in $f^{-1}(D_j(Z))$ and therefore $\mu(f^{-1}(D_j(Z)))<\infty$. 
Then, $G=\bigcup_{\ge 2} V_j=(\text{int}(V_1))^c$ satisfies $\mu(f^{-1}(G)<\infty$, we can proceed as before to prove that
$\mu(f^{-1}(\partial V_j))=0$ for all $j\ge 2$. Since $\partial V_1\subset \bigcup_{j\ge 2} \partial V_j$, we 
also conclude $\mu(f^{-1}(\partial V_1))=0$ and the result is shown.
\end{proof}

\section{Approximate-compactness, proof of Theorem \ref{the:mainresult II}}
\label{sec:proof main II}

Before starting with the proof of Theorem \ref{the:mainresult II}, we would like to remind some results about the RNP. Let $X$ be a Banach space, $(\Omega,\F,\mu)$ a $\sigma$-finite measure space and $1\leq p<\infty$. It is well-known that $X^*$ has the RNP with respect to  $(\Omega,\F,\mu)$ if and only if $(L^p(\Omega,\F,\mu,X))^*=L^q(\Omega,\F,\mu,X^*)$ where $q$ is such that $\frac1p+\frac1q=1$ (see Theorem 1.3.10 in \cite{martingale}). Moreover, in case $1<p<\infty$, the $\sigma$-finiteness condition can be dropped: if $X^*$ has the RNP then $(L^p(\Omega,\F,\mu,X))^*=L^q(\Omega,\F,\mu,X^*)$ for any measure space $(\Omega,\F,\mu)$ (see Corollary 1.3.13 in \cite{martingale}). This result will be used in the following proof. 

\begin{proof} ({\bf Theorem \ref{the:mainresult II}}) 
Consider a Banach space $Y$ with the RNP such that $X=Y^*$. It follows that $(L^q(Y))^*=L^p(Y^*)=L^p(X)$, where $q$ is such that $\frac1p+\frac1q=1$.

Assume $f\in L^p(X)$ and $(h_n)_n$ is a minimizing sequence. If 
$f\in \mathscr{G}_{p,k}(X)$, then clearly $\|f-h_n\|_p\to 0$ and the result holds in this case. 
So, in what follows we assume $f\notin \mathscr{G}_{p,k}(X)$. According to Proposition \ref{pro:bounded}, we can assume that $(h_n)_n$ is uniformly bounded.
Consider a reduced form of $h_n$ given by
$$
h_n=\sum\limits_{i=1}^{\ell(n)} x_{i,n} \ind_{A_{i,n}}.
$$
Passing to a subsequence if necessary, we can assume that $\ell(n)=\ell\le k$. By $w^*$-sequential compactness and taking a subsequence if necessary, we can suppose that
$$
x_{i,n}\xrightarrow[]{w^*} x_i.
$$
Let ${\bf y}=\{y_1,...,y_s\}$ the set formed by the different elements in $x_1,...,x_\ell$. For $1\leq j\leq s$ and $n\in\mathbb N$, let $I_j=\{i: x_{i,n}\xrightarrow[]{w} y_j\}$ and $B_{j,n}=\bigcup\limits_{i\in I_j} A_{i,n}$. Then, for all $n\in\mathbb N$, we define
$$
g_{n}=\sum\limits_{j=1}^s y_j \ind_{B_{j,n}}\in \mathscr{G}_{p,s}(X).
$$
The proof of Theorem \ref{the:mainresult I} shows that $(g_n)_n$ admits a minimizing subsequence. Consider  $\VS({\bf y})=\{V_j\}_{1\le j \le s}$ the Voronoi cells associated to
${\bf y}$. The proof of Proposition \ref{pro:simple special} shows that
$$
g=\sum\limits_{j=1}^s y_j \ind_{G_j}\in \mathscr{G}_{p,s}(X)
$$
is a minimizer in special $f$-Voronoi form, where $G_1=V_1$ and $G_j=V_j\setminus\bigcup_{m<j}G_m$ for $2\leq j\leq s$. According to Proposition \ref{the:special Voronoi}, we have $g$ is in simple special $f$-Voronoi form and $s=\ell=k$. It follows that $x_1,...,x_k$ are all different and
$$
h_n=\sum\limits_{i=1}^k x_{i,n} \ind_{A_{i,n}},\ \, g_n=\sum\limits_{i=1}^k x_{i} \ind_{A_{i,n}}\ \text{ and } \, 
g=\sum\limits_{i=1}^k x_{i} \ind_{D_i},
$$
where $D_i=\hbox{int}(V_i)$ for all $1\leq i\leq k$ (since $\mu(\partial V_i)=0$). To conclude, we will show that $h_n\xrightarrow[]{w^*}g$ in $L^p(X)$. So let $\wp\in L^q(Y)$ and let us prove that $\langle\wp,h_n-g\rangle\to 0$. For all $w\in\Omega$, we have
$$
\langle\wp(w),h_n(w)-g(w)\rangle=\sum\limits_{i,j} \langle\wp(w),x_{j,n}-x_i\rangle \ind_{A_{j,n}\cap D_i}(w), 
$$
showing that
\begin{align*}
\langle\wp,h_n-g\rangle=\int |\langle\wp(w),h_n(w)-g(w)\rangle| \ \mu(dw)&=
\sum\limits_i \int_{A_{i,n}\cap f^{-1}(D_i)} |\langle\wp(w),x_{i,n}-x_i\rangle| \ \mu(dw) \\
&+\sum\limits_{i\neq j} \int_{A_{j,n}\cap f^{-1}(D_i)} |\langle\wp(w),x_{j,n}-x_i\rangle| \ \mu(dw) \\
&:=\alpha_n+\beta_n.
\end{align*}
To conclude, we need to show that $\alpha_n\to 0$ and $\beta_n\to 0$.

\underline{\textbf{Case 1:}} $\mu(\Omega)<\infty$. Let us prove that $\alpha_n\to 0$. For every $i$ and $w\in\Omega$, $|\langle\wp(w),x_{i,n}-x_i\rangle| \ind_{A_{i,n}\cap f^{-1}(D_i)}(w)$ converges to $0$ because
$x_{i,n}\xrightarrow[]{w} x_i$.  This sequence is dominated by
$$
 |\langle\wp(w),x_{i,n}-x_i\rangle| \ind_{A_{i,n}\cap f^{-1}(D_i)}(w)\le \|\wp(w)\|_{X^*}\|x_{i,n}-x_i\|\ind_{f^{-1}(D_i)}(w)
 \le 2M \|\wp(w)\|_{X^*} \ind_{f^{-1}(D_i)}(w), 
$$
where $M=\sup\limits_n\max\{\|x_{i,n}\|: 1\le i\le k\}$. The function $\|\wp\|_{X^*} \ind_{f^{-1}(D_i)}$ 
belongs to $L^1$, because $\|\wp\|_{X^*}\in L^q$ and $f^{-1}(D_i)$ has finite measure. So the dominated convergence theorem implies that 
$$
\sum\limits_i \int_{A_{i,n}\cap f^{-1}(D_i)} |\langle\wp(w),x_{i,n}-x_i\rangle| \ \mu(dw)\to 0,
$$
from which we deduce that $\alpha_n\to 0$.

Now we show that $\beta_n\to 0$. Fix $1\le i\le k$ and $j\neq i$. First, we will prove that $\lim\limits_{n\to \infty} \mu(A_{j,n}\cap f^{-1}(D_i))=0$. 
For $y\in D_i$, it holds
$$
\|x_i-y\|<\|x_j-y\|,
$$
which proves that $D_i=\bigcup_{s\ge 0} C_s$, where
$$
C_0=\{y\in D_i: \|x_j-y\|^p-\|x_i-y\|^p>1\}$$
$$C_s=\{y\in D_i: 2^{-s+1}\ge \|x_j-y\|^p-\|x_i-y\|^p>2^{-s}\}$$for all $s\geq 1$. Consider the sequence
\begin{equation}
\label{eq:hhat}
\hat h_n=x_i \ind_{A_{j,n}\cap f^{-1}(D_i)}+ x_j \ind_{A_{j,n}\setminus f^{-1}(D_i)}+
\sum\limits_{l\neq j} x_l\ind_{A_{l,n}}.
\end{equation}
Since for all $w\in A_{j,n}\cap f^{-1}(D_i)$, we have $\|f(w)-x_j\|\ge \|f(w)-x_i\|$, we deduce
that $(\hat h_n)_n$ is also a minimizing sequence and moreover
\begin{equation}
\label{eq:h vs hhat}
\|f-h_n\|_p^p= \|f-\hat h_n\|_p^p+\int_{A_{j,n}\cap f^{-1}(D_i)} \|f(w)-x_j\|^p-\|f(w)-x_i\|^p \ \mu(dw).
\end{equation}
By decomposing the last integral according to $F_t=\bigcup_{0\le s\le t} C_s$, we obtain that
\begin{equation}
\label{eq:F_t}
2^{-t}\mu(A_{j,n} \cap f^{-1}(F_t))\le \int_{A_{j,n}\cap f^{-1}(D_i)} 
\|f(w)-x_j\|^p-\|f(w)-x_i\|^p \ \mu(dw)\to 0,
\end{equation}
and then
$$
\lim\limits_{n} \mu(A_{j,n}\cap f^{-1}(F_t))=0,
$$
showing that
$$
\limsup\limits_{n\to \infty} \mu(A_{j,n}\cap f^{-1}(D_i))\le \lim\limits_{t\to \infty}\mu(f^{-1}(F_\infty\setminus F_t))=0,
$$
and then for all $j\ne i$
$$
\lim\limits_{n\to \infty} \mu(A_{j,n}\cap f^{-1}(D_i))=0.
$$
Now, using H\"older's inequality, we have 
$$
\int_{A_{j,n}\cap D_i} |\langle\wp(w),x_{j,n}-x_i\rangle| \ \mu(dw)\le 2M \|\wp\|_q \left(\mu(A_{j,n}\cap D_i)\right)^{1/p}\to 0
$$
and it follows that $\beta_n\to 0$. This finishes the proof when $\mu$ is finite.

\underline{\textbf{Case 2:}} $\mu(\Omega)=\infty$. We assume that $x_1=0$ and $x_{1,n}=0$ for all $n\in\mathbb N$ (see Remark \ref{rem:1}). Recall that for all $j$, we have $\mu(\partial V_j)=0$ and also
$$
0<\mu(f^{-1}(D_i))<\infty,\ \, 0<\mu(A_{i,n})<\infty \text{ for all } 2\leq i\leq k \hbox{ and } \mu(f^{-1}(D_1))=\mu(A_{1,n})=\infty.
$$

Let us show that $\alpha_n\to 0$. Since $x_{1,n}=x_1=0$, we have
$$
\alpha_n=\sum\limits_{i\ge 2} \int_{A_{i,n}\cap f^{-1}(D_i)} |\wp(w)(x_{i,n}-x_i)| \ \mu(dw).
$$
Notice that $\mu(f^{-1}(\cup_{i\ge 2}D_i))=\mu(f^{-1}(D_1^c))\le \mu(f^{-1}(B(0,\e)^c))<\infty$. Then, using the same arguments as before, we deduce that $\alpha_n\to 0$.

To prove that $\beta_n\to 0$, we will show that for all $i\neq j$ it holds
$$
\lim\limits_{n\to \infty}\mu(A_{j,n}\cap f^{-1}(D_i))=0.
$$
The argument is the same as above when $i\ge 2$. Now suppose that $i=1$ and $j\neq 1$. Since $x_1=0\in D_1$, there exists $\varepsilon>0$ such that $B(0,\e)\subset D_1$. We can also assume that $d(0,\partial V_1)> 3\e$. In this way, we have $\|y-x_j\|\ge 2\e$ for all $y\in B(0,\e)$.
The Markov's inequality implies
$$
\e^p  \mu(f^{-1}((B(0,\e))^c)) \le \int_{\|f\|\ge \e} \|f(w)\|^p \ \mu(dw)\le \|f\|_p^p, 
$$
showing that $\mu(f^{-1}((B(0,\e))^c))<\infty$. Define $\hat h_n$ as in  \eqref{eq:hhat}. Again, $(\hat h_n)_n$ is a minimizing sequence and from 
\eqref{eq:h vs hhat}, we get
$$
\lim\limits_n \int_{A_{j,n}\cap f^{-1}(D_1)} \|f(w)-x_j\|^p-\|f(w)\|^p \ \mu(dw)=0.
$$
Notice that if $\|f(w)\|\le \e$ then $\|f(w)-x_j\|\ge 2\e$, this implies
$$
\int_{A_{j,n}\cap f^{-1}(B(0,\e))} \|f(w)-x_j\|^p-\|f(w)\|^p \ \mu(dw)\ge (2^p-1)\e^p \mu(A_{j,n}\cap f^{-1}(B(0,\e))),
$$
which implies that $\mu(A_{j,n}\cap f^{-1}(B(0,\e)))\to 0$. Now consider $$G_t=\{z\in D_1\setminus B(0,\e): \|z-x_j\|^p-\|z\|^p\ge 2^{-t}\}$$ for all $t\geq 0$. 
We point out that $G_t\uparrow G_\infty:= D_1\setminus B(0,\e)$ as $t\uparrow \infty$ and $\mu(f^{-1}(G_\infty))<\infty$. On the other hand, for each fixed $t\ge 1$, 
it holds (see also \eqref{eq:F_t})
$$
2^{-t}\mu(A_{j,n} \cap f^{-1}(G_t))\le \int_{A_{j,n}\cap f^{-1}(D_1)} \|f(w)-x_j\|^p-\|f(w)\|^p \ \mu(dw)\to 0,
$$
and then 
$$
\limsup\limits_{n\to \infty} \mu(A_{j,n} \cap f^{-1}(D_1))=\limsup\limits_{n\to \infty}
 \mu(A_{j,n} \cap f^{-1}(D_1\setminus B(0,\e)))\le \lim\limits_{t\to \infty} \mu(f^{-1}(G_\infty\setminus G_t))=0.
$$
So, we have proved that for all $j\neq i$ $$
\lim\limits_{n\to \infty} \mu(A_{j,n}\cap f^{-1}(D_i))=0.
$$ By H\"older's inequality, it follows that
$$
\beta_n=\sum\limits_{i\neq j} \int_{A_{j,n}\cap D_i} |\wp(w)(x_{j,n}-x_i)| \ \mu(dw)
\le 2M \|\wp\|_q \left(\sum\limits_{i\neq j} \mu(A_{j,n}\cap D_i)\right)^{1/p} \to 0
$$
and the proof is complete.
\end{proof}

\begin{proof} (Corollary \ref{cor:norm-compact})
Assume $f\in L^p(X)$ and $(h_n)_n$ be a minimizing sequence. Since $X$ is uniformly convex, 
it is reflexive (see Theorem 9.11 in \cite{BST} for example) and then fulfills the hypothesis 
of the previous theorem. Then, passing to a subsequence if necessary, it exists $g\in L^p(X)$ such that 
$h_n\xrightarrow[]{w} g$ and $g$ is a minimizer for $f$. In particular, we have that $\|h_n-f\|_p\to\|g-f\|_p$. 
Since $X$ is uniformly convex, then $L^p(X)$ is also uniformly convex and 
it follows that $(h_n-f)_n$ converges in $L^p(X)$ to $g-f$ (see Proposition 3.32 in \cite{Brezis}), 
showing that $(h_n)_n$ converges in $L^p(X)$ to $g$.
\end{proof}

\appendix
\section{The $p$-means}
\label{app:pmeans}

\begin{proposition}
Let $(\Omega,\F,\mu)$ be a measure space, $X$ a Banach space, $f\in L^p(X)$. Then:
\begin{enumerate}
    \item[(a)] the function $M_p(f,A)$ is convex, coercive and continuous in the following cases:
    \begin{enumerate}
        \item[(i)] $p=\infty$;
        \item[(ii)] $p<\infty$ and $\mu(A)<\infty$
    \end{enumerate}
    \item[(b)] if $p<\infty$ and $\mu(A)=\infty$, then $M_p(f,A)(x)=\infty$ for all $x\neq 0$ and $M_p(f,A)(0)\le \|f\|_p^p$
\end{enumerate}
\end{proposition}

\begin{proof}
$(a)$ Convexity follows from the convexity of the norm on $X$ and
convexity of the $p$-th power
\begin{equation}
\label{eq:convex pth mean}
\begin{array}{ll}
M_p(f,A)(\alpha x+(1-\alpha) y)\hspace{-0.2cm}&\le \int_A (\alpha \|f(w)-x\|+(1-\alpha)\|f(w)-y\|)^p \ \mu(dw)\\
\\
\hspace{-0.2cm}&\le \int_A \alpha \|f(w)-x\|^p+(1-\alpha)\|f(w)-y\|^p \ \mu(dw)\\
\\
\hspace{-0.2cm}&= \alpha M_p(f,A)(x)+(1-\alpha)M_p(f,A)(y).
\end{array}
\end{equation}
On the other hand, we have that
\begin{equation}
\label{eq:coer_p}
\|x\|\mu(A)^{1/p}\le \|(f-x)\ind_A\|_p+\|f\|_p=\big(M_p(f,A)(x)\big)^{1/p}+\|f\|_p,
\end{equation}
showing that $M_p(f,A)$ is coercive. The case of $p=\infty$ follows similarly and
moreover
\begin{equation}
\label{eq:coer_infty}
\|x\|\le M_\infty(f,A)(x)+\|f\|_\infty.
\end{equation}
Let us show that $M_p(f,A)$ is continuous. In fact, consider $x,y \in X$, we have for $p<\infty$
$$
\left|\left(M_p(f,A)(x)\right)^{1/p}-\left(M_p(f,A)(y)\right)^{1/p}\right|=\big|\|(f-x)\ind_A\|_p-\|(f-y)\ind_A\|_p\big|
\le \|x-y\|\mu(A)^{1/p} 
$$
and for $p=\infty$
$$
\big|M_\infty(f,A)(x)-M_\infty(f,A)(y)\big|=\big|\|(f-x)\ind_A\|_\infty-\|(f-y)\ind_A\|_\infty\big|\le \|x-y\|,
$$
showing the desired continuity.

$(b)$ When $\mu(A)=\infty$ and $p<\infty$, we have $M_p(f,A)(0)=\int_A \|f(w)\|^p \ \mu(dw)\le \|f\|_p^p$
and from \eqref{eq:coer_p}
$$
\forall x\neq 0: \hspace{0.1cm} M_p(f,A)(x)=\infty.
$$
\end{proof}

We summarize some properties of the $p$-th means in the following result

\begin{proposition} Let $(\Omega,\F,\mu)$ be a measure space, $X$ a Banach space, $f\in L^p(X)$ and $A\in\F$ such that $\mu(A)>0$ and, if $p<\infty$, we also assume that $\mu(A)<\infty$. 
\begin{enumerate} 
\item[(a)] For all $\e>0$ the set of $\e$-$p$-th means is a not empty closed, convex and bounded set. 
\item[(b)] The set of $p$-th means is closed, bounded and convex.
\item[(c)] Assume $X$ is a dual space with $w^*$-sequentially compact unit ball and $p<\infty$ 
then there exists a $p$-th mean for $f\in L^p$ in $A$. 
\item[(d)] Suppose that $1<p<\infty$. If the norm $\|\hspace{0.1cm}\|$ in $X$ is strictly convex or if $\mu\circ f^{-1}$ 
satisfies: for all $x\in X$
and all $\rho>0$ 
$$
\mu(f^{-1}(B(x,\rho)))>0,
$$
Then, there is at most one $p$-th mean for $f$ in $A$.
\end{enumerate}
\end{proposition}

\begin{proof}
$(a)$ By definition, for every $\e>0$ the set $C_\e=\{x \hbox{ is an $\e$-$p$-th mean}\}$ is nonempty.
On the other hand $C_\e$ is convex and closed, because $M_p(f,A)$ is convex and continuous. 
Since $M_p(f,A)$ is coercive we conclude $C_\e$ is bounded (see \eqref{eq:coer_p} and \eqref{eq:coer_infty}). 

$(b)$ The set $C_0=\cap_{\e>0} C_\e$ is the set of $p$-th means for $f$ in $A$. Therefore, it is closed
convex and bounded. 

$(c)$ Assume that $(x_n)_n$ is a $1/n$-$p$-th mean. 
We consider a subsequence $(x_{n_k})_k$ 
that converges in the $w^*$-topology. By Fatou's Theorem we have 
$$
\int_A \lim\inf_k \|f(w)-x_{n_k}\|^p \ \mu(dw)\le \lim\inf_k \int_A \|f(w)-x_{n_k}\|^p \ \mu(dw)=\underline M_p(f,A),
$$
and for all $w$ it holds $\|f(w)-x\|^p\le \lim\inf_k \|f(w)-x_{n_k}\|^p$ showing that $x$ is a $p$-th mean and 
$C_0$ is not empty. 

$(d)$ According to \eqref{eq:convex pth mean}, if $x,y$ are two possible $p$-th means 
and $0<\alpha<1$ then $\alpha x+(1-\alpha)y$ is also a $p$-th mean and we have equalities
in \eqref{eq:convex pth mean}.  Then, for almost all $w\in A$ it holds
$$
\|f(w)-x\|=\|f(w)-y\|, 
$$
because the $p$-th power is strictly convex. Also, we get from \eqref{eq:convex pth mean} that
for almost all $w\in A$ it holds
$$
\|f(w)-(\alpha x+(1-\alpha)y\|=\alpha \|f(w)-x\|+(1-\alpha)\|f(w)-y\|=\|f(w)-x\|=\|f(w)-y\|,
$$
which is not possible when $x\neq y$ if we used the strictly convex property of the norm in 
$X$ to the sphere $\{z\in X: \|f(w)-z\|=r\}$ with $r=\|f(w)-x\|=\|f(w)-y\|$. This shows that $x=y$.
\medskip 

Now, we assume $f$ and $\mu$ satisfy for all $z\in X$ and all $\rho>0$ 
$$
\mu(f^{-1}(B(z,\rho)))>0.
$$
If we assume $x\neq y$, we take $\rho=\frac13 \|x-y\|>0$. Then, for almost all $w \in  f^{-1}(B(x,\rho))$ we have
$\|f(w)-x\|< 1/3 \|x-y\|=\rho$ and
$$
\|x-y\|\le \|f(w)-x\|+\|f(w)-y\|< 1/3 \|x-y\|+\|f(w)-y\|,
$$
showing that
$$
\|f(w)-y\|\ge 2/3 \|x-y\|\ge 2\rho> 2 \|f(w)-x\|,
$$
which is a contradiction, and therefore $x=y$.
\end{proof}

\begin{remark} If $\Omega$ is a topological space with the Borel $\sigma$-field $\mathcal B(\Omega)$,
$\mu$ a Radon measure on $\mathcal B(\Omega)$ and $X$ a dual Banach space. Suppose that $f:\Omega\to X$ is lower-semicontinuous. Then the $p$-th mean of $f$ exists. In fact, by Proposition 7.12 in \cite{Folland}, the monotone convergence theorem, and then Fatou's lemma, holds for nets of positive lower-semicontinuous functions. Following the argument of the previous proof, since the unit ball of $X$ is $w^*$-compact, we deduce that $f$ admits a $p$-th mean.
\end{remark}

\begin{corollary}
\label{cor:charac_pmean}
Let $(\Omega,\F,\mu)$ be a measure space. Assume that $X$ is a Banach space with a G\^ateaux differentiable 
norm $\|\hspace{0.05cm}\|$. Let $f\in L^p(X)$ with $1<p<\infty$ and $A\in\F$ such that $0<\mu(A)<\infty$. We assume that $x$ is $p$-th mean of $f$ in $A$. 
Then, for all $v\in X$, the real function
$$
R(a)=M_p(f,A)(x+av)=\int_A \|f(w)-(x+av)\|^p \ \mu(dw),
$$ 
is convex and differentiable, with a minimum at $a=0$, for which
$$
0=R'(0)=\int_A p \|f(w)-x\|^{p-1} \partial(\|f(w)-x)\|)(v)\ \mu(dw).
$$
\end{corollary}
\begin{proof} Clearly $R$ is a continuous convex function. Let us prove that $R$ is differentiable. This
will follows from the dominated convergence theorem. Since the norm is G\^ateaux differentiable, for every $z\neq 0$, we have
the existence of 
$$
\partial(\|z\|)(v)=\lim\limits_{t\to 0} \frac{\|z+tv\|-\|z\|}{t}.
$$
showing that 
$$
\lim\limits_{t\to 0} \frac{\|z+tv\|^p-\|z\|^p}{t}=p\|z\|^{p-1}\partial(\|z\|)(v)
$$
Moreover, for $z=0$
$$
\lim\limits_{t\to 0} \frac{\|0+tv\|^p-\|0\|^p}{t}=0,
$$
showing that the function $\|z\|^p$ has a G\^ateaux derivative at any $z$.
Now, we need to show a domination, for that we use, for $a\ge 0, b\ge 0$
$$
|b^p-a^p|=p\left|\int_a^b s^{p-1} ds\right|\le p(a\vee b)^{p-1} |b-a|,
$$
which for $a=\|f(w)-x\|, b=\|f(w)-(x+tv)\|$ gives, using that $b\le a+|t|$ and $|b-a|\le |t|$, and for $|t|\le 1$,
$$
\left|\frac{\|f(w)-(x+tv)\|^p-\|f(w)-x\|^p}{t}\right|\le p(\|f(w)-x\|+1)^{p-1},
$$
which integrable using H\"older's inequality. This shows that $R$ is differentiable at $a=0$ and
$$
R'(0)=\int_A p \|f(w)-x\|^{p-1} \partial(\|f(w)-x)\|)(v)\ \mu(dw).
$$
Similarly it is shown that $R$ is differentiable and 
$$
R'(a)=\int_A p \|f(w)-(x+av)\|^{p-1} \partial(\|f(w)-(x+av))\|)(v)\ \mu(dw).
$$
Since $R$ has a minimum at $a=0$ it follows that $R'(0)=0$.
\end{proof}

\section*{Acknowledgements}

G. Grelier was supported by the grant PID2021-122126NB-C32 funded by MCIN/AEI/10.13039/ 501100011033 and by ``ERDF A way of making Europe'', by the European Union and by the grant 21955/PI/22 funded by Fundaci\'on S\'eneca Regi\'on de Murcia; and by MICINN 2018 FPI fellowship with reference PRE2018-083703, associated to grant MTM2017-83262-C2-2-P. J. San Mart\'in was partially founded by Basal FB 210005 ANID.

\end{document}